\documentclass[reqno,12pt]{amsart} 
\usepackage{amsmath,amssymb,amsthm,enumerate,mathrsfs,colonequals}
\usepackage{fullpage,url,tikz,xspace,setspace}
\usepackage{microtype}
\usepackage{calc}
\usepackage[alphabetic,lite,nobysame]{amsrefs} % for bibliography

\usepackage[all,cmtip]{xy}

\usepackage{color}
\usepackage{graphicx}
\newcommand{\Bmu}{\mbox{$\raisebox{-0.59ex}
  {$l$}\hspace{-0.18em}\mu\hspace{-0.88em}\raisebox{-0.98ex}{\scalebox{2}
  {$\color{white}.$}}\hspace{-0.416em}\raisebox{+0.88ex}
  {$\color{white}.$}\hspace{0.46em}$}{}}

\usepackage{capt-of}
\usepackage{graphicx}
\usepackage{hyperref}

\DeclareFontFamily{U}{wncy}{}
\DeclareFontShape{U}{wncy}{m}{n}{<->wncyr10}{}
\DeclareSymbolFont{mcy}{U}{wncy}{m}{n}
\DeclareMathSymbol{\Sha}{\mathord}{mcy}{"58}
\DeclareMathOperator{\Ad}{{\rm Ad}}

\DeclareMathOperator{\SL}{SL}
\DeclareMathOperator{\SLT}{\SL_{2}}

\DeclareMathOperator{\Ss}{S}
\DeclareMathOperator{\Pro}{P\!}

\usepackage{amssymb,amsmath}
\usepackage{amsfonts}
\usepackage{tkz-graph}
\newcounter{ctfig}

\newcommand{\Z}{{\mathbb Z}}
\newcommand{\G}{\mathcal{G}}
\newcommand{\Gn}{\mathcal{G}_n}

\newcommand{\Q}{{\mathbb Q}}
\newcommand{\cS}{\mathcal{S}}
\newcommand{\R}{{\mathbb R}}

\newcommand{\C}{\mathcal{C}}
\newcommand{\HH}{\mathcal{H}}

\newcommand{\sH}{\mathscr{H}}
\newcommand{\OO}{\mathcal{O}}
\newcommand{\OOn}{\OO_n}

\newcommand{\fp}{\mathfrak{p}}
\newcommand{\CC}{\mathbb C}

\DeclareMathOperator{\UT}{U_{2}}
\DeclareMathOperator{\Res}{Res}
\DeclareMathOperator{\GLT}{GL_{2}}
\DeclareMathOperator{\SLTh}{SL_{3}}
\DeclareMathOperator{\PUT}{PU_{2}}
\DeclareMathOperator{\SUT}{SU_{2}}

\DeclareMathOperator{\Disc}{Disc}
\DeclareMathOperator{\UTz}{U_2^\zeta}
\DeclareMathOperator{\SOT}{SO_{3}}

\DeclareMathOperator{\PUTz}{PU_2^\zeta}
\DeclareMathOperator{\PSUT}{PSU_{2}}
\DeclareMathOperator{\Tr}{Tr}
\DeclareMathOperator{\Nm}{N}
\DeclareMathOperator{\Sel}{Sel}
\DeclareMathOperator{\Cl}{Cl}
\DeclareMathOperator{\val}{val}

\newtheorem{theorem}{Theorem}[section]
\newtheorem{lemma}[theorem]{Lemma}
\newtheorem{proposition}[theorem]{Proposition}
\newtheorem{example1}[theorem]{Example}
\theoremstyle{example}
\theoremstyle{remark}
\theoremstyle{corollary}
\theoremstyle{definition}
\newtheorem{remark1}[theorem]{Remark}
\newtheorem{definition1}[theorem]{Definition}
\newtheorem{corollary}[theorem]{Corollary}

\begin{document}
%% \foreach \x in{graphics,floats}{%
%%     \immediate\write18{pdflatex -jobname=template-\x\space "\def\noexpand\placeholder{\x} \noexpand\input{template}"}%
%%     \includepdf[pages=-]{template-\x}%
%% }
\bibliographystyle{plain}
\bibstyle{plain}

\title[The Clifford-cyclotomic group]
{The Clifford-cyclotomic group and Euler-Poincar\'{e} characteristics}

\author{Colin Ingalls \and Bruce W. Jordan \and Allan Keeton \and \\
Adam Logan \and Yevgeny Zaytman}

\address{School of Mathematics and Statistics, Carleton University, Ottawa, ON K1S 5B6, Canada}
\email{cingalls@math.carleton.ca}

\address{Department of Mathematics, Box B-630, Baruch College,
The City University of New York, One Bernard Baruch Way, New York
NY 10010}
\email{bruce.jordan@baruch.cuny.edu}

\address{Center for Communications Research, 805 Bunn Drive,
Princeton, NJ 08540}
\email{agk@idaccr.org}

\address{The Tutte Institute for Mathematics and Computation,
  P.O. Box 9703, Terminal, Ottawa, ON K1G 3Z4, Canada}
\address{School of Mathematics and Statistics, 4302 Herzberg
  Laboratories, 1125 Colonel By Drive, Ottawa, ON K1S 5B6, Canada}
\email{adam.m.logan@gmail.com}

\address{Center for Communications Research, 805 Bunn Drive,
Princeton, NJ 08540}
\email{ykzaytm@idaccr.org}

\subjclass[2010]{Primary 81P45; Secondary 20G30}
\keywords{Clifford group, T gate, Clifford cyclotomic, Euler-Poincar\'{e} characteristics}

\begin{abstract}
For an integer $n\geq 8$ divisible by $4$, let $R_n=\Z[\zeta_n,1/2]$ and
let $\UT(R_n)$ be the group of $2\times 2$ unitary matrices with entries in $R_n$. Set
\mbox{$\UTz(R_n)=\{\gamma\in\UT(R_n)\mid \det\gamma\in\langle\zeta_n\rangle\}$}.
Let $\Gn\subseteq \UTz(R_n)$ be the Clifford-cyclotomic group generated
by a Hadamard matrix $H=\frac{1}{2}[\begin{smallmatrix}
1+i & 1+i\\1+i &-1-i\end{smallmatrix}]$ and the gate
$T_n=[\begin{smallmatrix}1 & 0\\0 & \zeta_n\end{smallmatrix}]$.
We prove that $\Gn=\UTz(R_n)$ if and only if $n=8, 12, 16, 24$
and that $[\UTz(R_n):\Gn]=\infty$ if $\UTz(R_n)\neq \Gn$.
We compute the Euler-Poincar\'{e} characteristic of the groups
$\SUT(R_n)$, $\PSUT(R_n)$, $\PUT(R_n)$, $\PUTz(R_n)$, and $\SOT(R_n^+)$.

\end{abstract}

\maketitle

\section{Introduction}
\label{taco}
Let $\UT = \{ g \in \GLT(\CC)\mid gg^\dagger = 1\}$ be the group of $2\times 2$ unitary matrices stabilizing the standard hermitian form on $\CC^2$ with $\dagger$
denoting conjugate-transpose. Let 
$\UTz$ and $\SUT$ be its subgroups of matrices whose determinants are roots of unity or 1 respectively. 
For a subring  $R \subseteq \CC$, write $\UT(R):= \UT \cap \GLT(R)$
for the subgroup of $\UT$ whose matrix entries lie in $R$; similarly
$\UTz(R):= \UTz \cap \GLT(R)$,
 and $\SUT(R) := \SUT \cap \SLT(R)$.
Let \mbox{$\SOT=\{g\in \SLTh(\R) \mid g g^t=1\}$}.
For a subring $R^+\subseteq \R$, write $\SOT(R^+)$ for the subgroup of 
$\SOT$ whose entries lie in $R^+$.

Throughout this paper $n=2^s d$ is a positive integer with $d$ odd.
Unless explicitly stated otherwise, we assume
$s\geq 2$.
Let $\zeta_n:=e^{2\pi i/n}$, $K_n:=\Q(\zeta_n)$, and $R_n=\Z[\zeta_n,1/2]$. 
Set $F_n=K_n^+:=\Q(\zeta_n+\overline{\zeta_n})$ and 
$R_n^+=\Z[\zeta_n+\overline{\zeta}_n, 1/2]$.
Then $i\in R_n$ and $R_n=R_n^+ \oplus R_n^+i$ since $1/2\in R_n^+$.
The {\em Clifford group} $\C$ can be defined as 
$\C=\UT(R_4)$ \cite[Section 2.1]{fgkm}. 
Set
\begin{equation}
\label{snow}
T_n:=\left[\begin{array}{cc}
1 & 0\\
0 & \zeta_n\end{array}\right]\in \UT(\Z[\zeta_n])\subseteq \UT(R_n).
\end{equation}
Define the {\em Clifford-cyclotomic group} \cite[Section 2.2]{fgkm}
(resp., {\em special} Clifford-cyclotomic group) by
\begin{equation}
\label{sun}
\Gn=\langle \C, T_n\rangle\qquad\mbox{\textup{(}resp., $\Ss\! 
\Gn=\Gn\cap \SUT(R_n)$\textup{)};}
\end{equation}
we have $\Gn\subseteq \UTz(R_n)$.  In general, 
$\UTz(R_n)\subsetneq  \UT(R_n)$. For a subgroup $H\leq \UT(R_n)$, denote
by $\Pro H$ the image of $H$ in $\PUT(R_n)$.
The Adjoint map $\Ad:\SUT(R_n) \rightarrow \SOT(R_n^+)$ induces maps $\pi:\UT(R_n)\rightarrow
\SOT(R_n^+)$ and $\overline{\pi}:\PUT(R_n)\rightarrow \SOT(R_n^+)$; 
see Section \ref{yellow}.

Let $G(r,s)$ be the subgroup of $\SOT(\R)$ generated by rotations
of order $r$  and order $s$ about chosen perpendicular axes.  For an appropriate
choice of axes one has $G(4, n)\subseteq \SOT(R_n^+)$.
In Theorem \ref{amalgam} we show that $\pi(\Gn)=G(4,n)$.
The subgroup structures 
\begin{equation}
\label{rational}
\Gn\leq \UTz(R_n)\leq \UT(R_n), \quad G(4,n)\leq \SOT(R_n^+)
\end{equation}
play a large role in exact synthesis for quantum gates 
in single-qubit quantum computation.  The following
results are known:
\begin{theorem}
\label{gone}
\begin{enumerate}[\upshape (a)]
\item
\label{gone1}
$G(4,8)=\SOT(R_8^+)$ and $\G_8=\UT(R_8)$ \cites{s2,fgkm}, 
$G(4,12)=\SOT(R_{12}^+)$ and $\G_{12}=\UT(R_{12})$
\cites{s2, brs}, $G(4,16)=\SOT(R_{16}^+)$ and  $\G_{16}=\UT(R_{16})$ \cite{s2}, 
$G(4, 24)=\SOT(R_{24}^+)$ and $\G_{24}=\UT(R_{24})$
\cite{fgkm}.
\item
\label{gone2}
For an integer $n$ we have $\UTz(R_n)=\UT(R_n)$ if and only if
\[
-1\bmod d\in\langle 2\bmod d\rangle \leq (\Z/d\Z)^{\times}
\]
\cite[Theorem 5.3]{fgkm}.
\item
\label{gone3}
Let $S_4$ be the symmetric group
on $4$ letters and $D_m$ be the dihedral group
of order $2m$. We have $G(4,n)\cong S_4\ast_{D_4}D_n$ \cite{rs}.
\item
\label{gone4}
If $n=2^s$, $s\geq 5$, then 
$G(4,n)$ is of infinite index in $\SOT(R_n^+)$
\cite{s2}.
\end{enumerate}
\end{theorem}

Serre \cite{s2} introduced Euler-Poincar\'{e} characteristics to
the study of $G(4,n)$ and $\Gn$, as well as observing that $\SOT(R_n^+)$ for $n=2^s$ acts on a tree
by looking at it over $\Q_{2}$. 
Theorem \ref{gone}\eqref{gone4} follows from computing the Euler-Poincar\'{e}
characteristic $\chi$ of $G(4,n)$ and $\SOT(R_n^+)$ for $n=2^s\geq 8$:
\begin{theorem}
\label{gut}
\textup{(Serre \cite{s2})}
Suppose $n=2^s\geq 8$.
\begin{enumerate}[\upshape (a)]
\item
\label{gut1}
$\chi(G(4,n))=-1/12 +1/2n$.

\item
\label{gut2}
$\chi(\SOT(R_n^+))=-2^{-2^{s-2}}\zeta_{F_n}(-1)$.
\end{enumerate}
\end{theorem}

In this paper we prove the following theorem, settling 
affirmatively a conjecture of Sarnak
\cite[p. $15^{\rm IV}$]{sa}:
\begin{theorem}
\label{green}
Suppose $4|n$ with $n\geq 8$.  
\begin{enumerate}[\upshape (a)]
\item
\label{green1}
We have  $\Gn=\UTz(R_n)$ if and only if $n=8,12,16,24$.
\item
\label{green2}
We have $\Ss\! \Gn=\SUT(R_n)$ if and only if 
$n=8, 12, 16, 24$. 
\item
\label{green3}
We have $G(4,n)=\SOT(R_n^+)$ if and only if $n=8, 12, 16, 24$.
\end{enumerate}
In all cases above where there is not equality, the index is infinite.
\end{theorem}
\noindent We prove Theorem \ref{green} by computing Euler-Poincar\'{e}
characteristics with $4|n$, $n\geq 8$, generalizing Theorem \ref{gut}.
We prove that 
\[
\chi(\Ss\!\Gn)=\chi(G(4,n))=\chi(\Pro\Gn)
=-1/12+1/2n
\]
 in Theorem \ref{mouse}.
Then in Theorem \ref{plants} we compute $\chi$ of 
$\SUT(R_n)$, $\PSUT(R_n)$, $\PUT(R_n)$, $\PUTz(R_n)$, and $\SOT(R_n^+)$.
We gain a foothold on these Euler-Poincar\'{e} characteristics by
considering the group scheme $\SUT(\Z[1/2])$ over $\Z[1/2]$, denoted $A_1^\ast$.
We have $A_1^\ast(\R)=\SUT(\CC)$ and $A_1^\ast(R_n^+)=\SUT(R_n)$.
The results of Serre \cite{s1} (which depend on theorems of Harder)
 apply to compute $\chi(\SUT(R_n))$ 
because because $A_1^\ast$ is simply 
connected and simple. We then deduce $\chi$ of the other groups from
this using properites of Euler-Poincar\'{e} characteristics.
The relationship between $\chi(\PUT(R_n))$ and $\chi(\SOT(R_n^+))$
is particularly interesting -- it involves embedding $\PUT(R_n)$
in $\SOT(R_n^+)$ via the Adjoint representation with attendant invariant $\overline{c}(R_n)$ defined in
Definition \ref{mare}\eqref{mare1}.

\section{The special Clifford-cyclotomic group}
\label{veal}

For a complex number $z$ of absolute value $1$, define the 
unitary matrix
\begin{equation}
\label{hidden}
H(z)=\frac{1}{2}\left[\begin{array}{cc} 1+i & z(1+i)\\
\overline{z}(-1+i) & 1-i\end{array}\right]
\end{equation}
of determinant $1$.  In particular $H(1)\in\C$.  Following
\cite[(2)]{fgkm}, we take our Hadamard matrix to be 
\begin{equation}
\label{hidden2}
H:=\frac{1}{2}\left[\begin{array}{cc} 1+i & 1+i\\
1+i & -1-i\end{array}\right]\in\C.
\end{equation}
We have $
H=T_{4}^{-1}H(1)$
with $T_n\in\UT(R_n)$ as in \eqref{snow}
and 
%\begin{equation}
%\label{valley}
$
T_n^{-j}H(1)T_n^j=H(\zeta_n^j)\in\SUT(R_n)
$
for integers $j$ if $4|n$.  With $4|n$, set
\begin{equation}
\label{valley2}
\sH_n:= \langle H(\zeta_n), H(\zeta_n^2), \ldots , H(\zeta_n^{n-1}),
H(\zeta_{n}^n)=H(1)\rangle \leq \Ss\!\Gn\leq \SUT(R_n).
\end{equation}
\begin{proposition}
\label{valley3}
Assuming $4|n$, we have
\begin{enumerate}[\upshape (a)]
\item
\label{valley31}
$[\Gn:\Ss\!\Gn]=n$,
\item
\label{valley32}
$\Gn=\langle H, T_n\rangle = \langle H(1), T_n\rangle$,
\item
\label{valley33}
$\Ss\!\Gn=\sH_n$.
\end{enumerate}
\end{proposition}
\begin{proof}
\eqref{valley31} follows from the exact sequence
\begin{equation}
\label{valley4}
1\rightarrow \SUT(R_n)\rightarrow \UTz(R_n)\stackrel{\det}{\rightarrow}
\langle \zeta_n\rangle \rightarrow 1,
\end{equation}
since the roots of unity in $R_n$ are $\langle\zeta_n\rangle$ as
$n$ is even.

\eqref{valley32} is shown in \cite[Section 2.2]{fgkm}.

For \eqref{valley33}, let $w$ be a word in $H(1)$ and $T_n$ of determinant $1$ with 
$k$ occurences of 
$H(1)$.  We proceed  by induction on $k$.  If $k=0$,
then the word must be $1$.  If $k=1$, the word must be
$T_n^{-j} H(1) T_n^j = H(\zeta_n^j)$ for some $0\leq j\leq n$.  
Suppose inductively that every word in
$H(1)$ and $T_n$ of determinant $1$ with at most $k_0$ occurences of $H(1)$
is in $\sH_n$, and let $w$ be a word in which $H(1)$ appears $k_0+1$ times.
Choose $a$ with $0 \le a < n$ such that $w$ begins with $T_n^aH(1)$.  
Then $H(\zeta_n^{-a})^{-1}w\in \Ss\!\Gn$ 
has at most $k_0$ occurences of $H(1)$, and so is in $\sH_n$ by assumption.
Hence $w\in \sH_n$ and $\Ss\!\Gn=\sH_n$.
\end{proof}

\begin{theorem}
\label{mountain}
Assume $4|n$.  Then $\UTz(R_n)=\Gn$ if and only if $\SUT(R_n)=\Ss\!\Gn$.
\end{theorem}
\begin{proof}
First, suppose that $\SUT(R_n)=\sH_n$ and let
$\alpha\in\UTz(R_n)$.  Let $\det\alpha =\zeta_n^j$, where
$0\leq j<n$.  Then $\alpha=T^j\alpha '$, where $\det\alpha'=1$ and
so $\alpha '\in\SUT(R_n)$.  Since the generators of 
$\SUT(R_n)$ belong to $\langle H, T_n\rangle$, it follows that
$\alpha$ does too.

In the other direction, suppose that $\UTz(R_n)=\Gn$.  Then
$\SUT(R_n)=\Ss\!\Gn$ trivially by definition.
\end{proof}

\section{ \texorpdfstring{$\SUT(R_n)$}{SU2(Rn)} and 
\texorpdfstring{$\SOT(R_n^+)$}{SO3(Rn+)}}
\label{hind}

\begin{definition1}
\label{testy}
\begin{enumerate}[\upshape (a)]
\item
\label{testy1}
Throughout this paper $R^+$ is the ring of $S$-integers in a
totally real number field $F$, where $S$ contains the archimedean
places and all places above $2$. We put $R=R^+[i]$  and $K=F(i)$. Both $R^+$
and $R$ are Dedekind domains.
\item
\label{testy2}
Define $A_1^*$ to be the group scheme over $\Z[1/2]$  with
$$A_1^*(B) =\left\{ \left[\begin{matrix} a + bi & c+ di \\ -c + di & a-bi \end{matrix}\right] : a^2 + b^2 + c^2 +d^2 = 1;\,\, a,b,c,d\in B\right\}$$
for any $\Z[1/2]$-algebra $B$ with group operation defined by 
matrix multiplication.  In particular, 
$A_1^\ast(B)=\SUT(B[i])$.
For example, $A_1^*(\R) = \SUT(\CC)$. 
\end{enumerate}
\end{definition1}

By $\SOT$ we mean the group of $3\times 3$ matrices of determinant $1$ that stabilize the standard inner product on $\R^3$. 
It is defined as a group scheme over $\Z$ by $\det(g) =1$ and $gg^t=1$.  %so for $R^+ \subset \R$, the $R^+$-valued points in $\SOT$, $\SOT(R^+)$, has its usual meaning. 
There is an exact sequence of group schemes
\begin{equation}\label{ses}1 \rightarrow \Bmu_2 \rightarrow A_1^* \xrightarrow{\Ad} \SOT \rightarrow 1\, ,\end{equation}
given by $\SUT$ acting by conjugation on the three-dimensional real vector space $V$ of trace-$0$ $2\times 2$ hermitian ($m^\dagger = m$) matrices in the {\it Pauli basis} 
$$ \sigma_x = \left[\begin{array}{ll}0 & 1\\ 1 & 0\end{array}\right], \sigma_y = \left[\begin{array}{rr}0 & -i\\ i & 0\end{array}\right],\sigma_z = \left[\begin{array}{cc}1 & 0\\ 0 & -1\end{array}\right].$$ In terms of hermitian matrices the standard form is $\langle A, B\rangle =  \frac{1}{2} \Tr(AB)$:
$$\frac{1}{2}\Tr\left( \left[\begin{matrix} z & x-iy\\x+iy & -z \end{matrix}\right]^2\right) = x^2 + y^2 +z^2 ,$$
which is obviously preserved under conjugation by $\SUT$. This is the Adjoint action of $\SUT$ on its Lie algebra $iV$ of trace-$0$ skew-hermitian matrices in disguise.
Explicitly, we have \cite[Appendix A]{na}\begin{small}
\begin{equation}
\label{pi} 
\Ad  \left(\left[\begin{array}{cc} a + bi & c+ di \\ -c + di &
  a-bi \end{array}\right]\right)=\left[\begin{array}{ccc}a^2 -b^2 -c^2 + d^2 &
  2ab +2cd & -2ac +2bd \\ -2ab +2cd & a^2 -b^2 +c^2 -d^2 & 2ad +2bc
  \\ 2ac +2bd & -2ad + 2bc & a^2+b^2-c^2-d^2 \end{array} \right] .
\end{equation} \end{small}

The map $\Ad$ factors as 
\begin{equation}
\label{crow}
A_1^*(R^+)=\SUT(R) \twoheadrightarrow \PSUT(R) \hookrightarrow \SOT(R^+).
\end{equation}
The adjoint action $\Ad$ given in \eqref{pi} extends to a group homomorphism $\pi:\UT(R)\rightarrow \SOT(R^+)$
via conjugation on the trace-$0$ $2\times2$ hermitian matrices in the Pauli basis.  We have 
\[
\pi(g)=\Ad (\frac{1}{\sqrt{\det g}}g)
\]
for an arbitrary choice of $\sqrt{\det g}$.
The map $\pi$ in turn factors as
\begin{equation}
\label{finch}
\UT(R)\twoheadrightarrow \PUT(R)\stackrel{\overline{\pi}}{\hookrightarrow}\SOT(R^+).
\end{equation}
We view $\PSUT(R)$ as a subgroup of $\SOT(R^+)$ via \eqref{crow}
and we view $\PUT(R)$ as a subgroup of $\SOT(R^+)$ via \eqref{finch} with
$\PSUT(R)\leq \PUT(R)\leq \SOT(R^+)$.  
\begin{remark1}
\label{robin}
{\rm
In Section \ref{orange} we will define a map $\phi$ from $\SOT(R^+)$
into a finite elementary abelian $2$-group (the Selmer group $\Sel_2^+(R^+)$)
with kernel $\PSUT(R)$.  From this it follows that $\PSUT(R)$ and 
$\PUT(R)$ are {\em normal} subgroups of $\SOT(R^+)$, cf. Corollary \ref{foal}.
}
\end{remark1}

\section{$\SOT(R^+)/\PSUT(R)$ and $\SOT(R^+)/\PUT(R)$}
\label{orange}

The short exact sequence (\ref{ses}) remains short exact on $\R$-points 
$$1 \rightarrow \langle \pm 1\rangle \rightarrow A_1^*(\R) \xrightarrow{\Ad} \SOT(\R) \rightarrow 1$$
with $A_1^*(\R) = \SUT(\CC)$, but in general for $R^+$ 
we only have
\begin{equation}
\label{bud}
1 \rightarrow \Bmu_2(R^+)=\langle\pm 1\rangle  \rightarrow A_1^*(R^+) \xrightarrow{\Ad} \SOT(R^+)\, .
\end{equation}
In our situation $A_1^*(R^+)$ does not surject onto $\SOT(R^+)$.  In particular the map $\Ad$ factors as  (see the next section) and the map from $\SUT(R)$ to $\PUT(R)$ is not
surjective.  Indeed, for us $R$ is a localization of an order in a number field,
so the group of roots of unity of $R$ is finite, generated by some
root of unity
$\zeta$.  Then $\zeta$ is not a square in $R$, so
$\left[\begin{smallmatrix}\zeta&0\\0&1\end{smallmatrix}\right]$ is an element of $\PUT(R)$ whose
  determinant is not a square.  Therefore it cannot be the image of any element of
  $\SUT(R)$.  Since the map $\PUT(R) \to \SOT(R^+)$ is injective, this implies
  that $\SUT(R) \to \SOT(R^+)$ is not surjective either, proving the following proposition.
\begin{proposition}
\label{dint}
Let $R^+$ be the $S$-integers in a toally real field $F$, where
$S$ contains the archimedean primes and all primes above $2$, and 
let $R=R^+[i]$. Then the group $\SOT(R^+)/\PSUT(R)$ is nontrivial.
\end{proposition}

Even the map $\PUT(R) \to \SOT(R^+)$ may not be surjective.
\begin{example1}
\label{dreary}
The map
$\PUT(\Z[\sqrt{21}, i,1/2] )\hookrightarrow \SOT(\Z[\sqrt{21}, 1/2])$ 
is not surjective.\\
{\rm
Let $R^+=\Z[\sqrt{21},1/2]$ and $R=R^+[i]$.
Let $u=\frac{5+\sqrt{21}}{2}\in (R^+)^\times$ which is totally positive
and {\em not} the norm of a unit in $R$. (One checks that $R^\times$ is generated
by $u ,i, 1+i$ and hence that $u$ is not a norm from $R^\times$.) Choose 
$q\in\left(\frac{-1,-1}{R^+}\right)$ of norm $u$, such as
$\frac{4+\sqrt{21}+i+j+k}{4}$. The homomorphism from the unit Hamilton
quaternions over $R^+$  to $\SOT$ takes $q$ to $$
T_q=\left[\begin{array}{ccc}
\frac{\sqrt{21}+3}{8} & \frac{1}{4} & \frac{-\sqrt{21}+3}{8}\\[.045in]
\frac{-\sqrt{21}+3}{8} & \frac{\sqrt{21}+3}{8} &\frac{1}{4}\\[.045in]
\frac{1}{4} & \frac {-\sqrt{21}+3}{8} & \frac{\sqrt{21}+3}{8}
\end{array}\right]\in\SOT(R^+).
$$
The $2\times 2$ matrix $M$ corresponding to $q$ is
$$
M_q=\left[\begin{array}{cc} \frac{4+\sqrt{21}+i}{4}&
\frac{1+i}{4}\\[.045in]
\frac{-1+i}{4} & \frac{4+\sqrt{21}-i}{4}
\end{array}\right];
$$
it has the property that $M M^{\dagger}=u\,{\rm Id}_{2\times 2}$.
The element of $\PUT(\CC)$ mapping to $T_q$ is obtained by
dividing $M_q$ by an element of $\CC$ of norm $u$.  However,
lifting this element to an element of $\PUT(R)$ would require
finding an element of $R$ of norm $u$, which does not exist.
Hence $T_q\in\SOT(R^+)$ is not the image of any element of 
$\PUT(R)$.
}
\end{example1}

In this section we will prove that $\SOT(R^+)/\PSUT(R)$ and $\SOT(R^+)/\PUT(R)$ are finite
abelian $2$-groups, with $\SOT(R^+)/\PSUT(R)$ nontrivial by Proposition \ref{dint}.
Denote by $F_+$ the totally positive elements of $F$. For any
subset $S\subseteq F$, denote by $S_+\subseteq F_+$
the totally positive elements of $S$.

\begin{definition1}
\label{eel}
{\rm
 Let $M_x, M_y, M_z \in \SOT(R^+)$ be the diagonal matrices with entries
  $$(1,-1,-1), (-1,1,-1), (-1,-1,1)$$ respectively.
  
Define that following $R^+$-valued functions for $M\in\SOT(R^+)$:
\begin{align*}
\phi_1(M)&:=(1+M_{11}+M_{22}+M_{33})/4=(1+\Tr(M))/4\\
\phi_2(M)&:=(1-M_{11}-M_{22}+M_{33})/4=(1+\Tr(MM_z))/4\\
\phi_3(M)&:=(1-M_{11}+M_{22}-M_{33})/4=(1+\Tr(MM_y))/4\\
\phi_4(M)&:=(1+M_{11}-M_{22}-M_{33})/4=(1+\Tr(MM_x))/4
\end{align*}
and
\begin{align*}
\theta_{12}(M):=\theta_{21}(M)&:=(M_{12}-M_{21})/4\\
\theta_{13}(M):=\theta_{31}(M)&:=(M_{31}-M_{13})/4\\
\theta_{14}(M):=\theta_{41}(M)&:=(M_{23}-M_{32})/4\\
\theta_{34}(M):=\theta_{43}(M)&:=(M_{12}+M_{21})/4\\
\theta_{24}(M):=\theta_{42}(M)&:=(M_{31}+M_{13})/4\\
\theta_{23}(M):=\theta_{32}(M)&:=(M_{23}+M_{32})/4\\
\end{align*}
}
\end{definition1}
\begin{definition1}
\label{owl}
{\rm
Let $\sf{R}$ be the $S$-integers in a totally real number field $F$.
Define the Selmer group
\[
\Sel^+_2({\sf{R}}):=\{x\in F_+^\times\mid \val_{\fp}x\equiv 0\pmod{2}\text{ for every finite prime $\fp$ of $\sf{R}$}\}/(F^\times)^2.
\]
We denote by $\Cl({\sf R})$ the class group of $\sf{R}$.
}
\end{definition1}
It is not difficult to compute $\Sel^+_2({\sf R})$ in examples using the following elementary proposition.
\begin{proposition}
\label{hawk}
There is an exact sequence of abelian groups
\[
{\sf R}^\times\stackrel{2}{\rightarrow} {\sf R}^\times_+
\rightarrow \Sel_2^+({\sf R})\rightarrow \Cl({\sf R})\stackrel{2}{\rightarrow} \Cl({\sf R}).
\]
\end{proposition}
In particular, let $r=[F:\Q]$ and let $s$ be the number of finite
primes in $S$.
Then the kernel of the signature map ${\sf R}^\times \to (\Z/2\Z)^{r}$
is precisely ${\sf R}^\times_+$, while
${\sf R}^\times/({\sf R}^\times)^2 \cong (\Z/2\Z)^{r+s}$.  Thus if the
image of the signature map is isomorphic to $(\Z/2\Z)^v$, then
${\sf R}^\times_+/({\sf R}^\times)^2 \cong (\Z/2\Z)^{r+s-v}$.

%\bruce{We should explicitly say what $(R^+)^\times_+/[(R^+)^\times]^2$ is. 
%I guess this requires
%some restriction on $R$ -- like $R$ is the $S$-integers in a totally real 
%field with all the primes above $2$ in S.  But then this quotient should be
%given in terms of the cokernel of the signature map, I imagine...}
%\end{comment}
This makes it straightforward to compute the following examples:
\begin{proposition}
\label{grounded}
Let $R_n$, $R_n^+$ be as in the introduction.
\begin{enumerate}[\upshape (a)]
\item
\label{grounded1}
Suppose $n=2^s$, $n\geq 8$.
Then $\Sel_2^+(R_n^+)\cong \Z/2\Z$.
\item
\label{grounded2}
Suppose $n=3\cdot 2^s$, $4|n$.
Then $\Sel_2^+(R_n^+)\cong \Z/2\Z$.
\end{enumerate}
\end{proposition}
\begin{proof}
Let $\OOn:=\Z[\zeta_n+\overline{\zeta}_n]$, the ring of integers in $F_n:=\Q(\zeta_n)^+$.\\
\eqref{grounded1}: Let $n=2^s$, $n\geq 8$. Then $\OOn$ has odd class number by
\cite[Theorem 10.4(b)]{w} and so $R_n^+=\OOn[1/2]$ has odd class number.
Every totally positive unit in $\OOn$ is  a square by Weber's Theorem \cite{we} and there is one prime $\fp$ in
$F_n$ above $2$.  Hence $(R_n^+)_+/[(R_n^+)^\times]^2\cong \Z/2\Z\cong \Sel_2^+(R_n^+)$.\\
\eqref{grounded2}: Let $n=3\cdot 2^s$, $s\geq 3$. Then $\OOn$ has odd class number
by applying \cite[Theorem 10.4]{w} to $F_n/\Q(\sqrt{3})$. We have 
$(R_n^+)_+^\times/[(R_n^+)^\times]^2\cong \Z/2\Z$
by \cite[Theorem 3.13(b)]{IJKLZ1}.

\end{proof}
The functions of Definition \ref{eel} satisfy the following properties.
\begin{lemma}
\label{phi theta}
For $M\in\SOT(R^+)$ and $$A=\left[\begin{array}{cc} a_1 + a_2i & a_3+ a_4i \\ -a_3 + a_4i &
  a_1-a_2i \end{array}\right]\in\SUT(R)$$ we have:
\begin{enumerate}[\upshape (a)]
\item \label{1} $\phi_i(\Ad(A)) = a_i^2$, $1\leq i\leq 4$.
\item \label{2} $\theta_{ij}(\Ad(A)) = a_ia_j$, $1\leq i, j\leq 4$, $i\neq j$.
\item \label{3} $\phi_1(M)+\phi_2(M)+\phi_3(M)+\phi_4(M)=1$.
\item \label{4} $\phi_i(M)\phi_j(M) = \theta_{ij}(M)^2$, $1\leq i,j\leq 4$, $i\neq j$,
and 
 $\theta_{\pi(1)\pi(2)}(M)\theta_{\pi(3)\pi(4)}(M)$ does not depend on the choice
  of $\pi \in \cS_4$.
\item \label{5} There exists a well a unique well-defined function
  $\phi: \SOT(R^+)\to F^\times/(F^\times)^2$ which agrees with
  each $\phi_i$, $1\leq i\leq 4$, whenever the latter is nonzero.
\item \label{6} The image of $\phi$ lies in $\Sel^+_2(R^+)\subseteq F^\times/(F^\times)^2$.
  In other words
  $\phi(M)$ has even valuation at all primes of $R^+$ and is totally
  positive.
\item \label{7} For $i \in \{x,y,z\}$ we have $\phi(M) = \phi(MM_i) = \phi(M_iM)$.
\end{enumerate}
\end{lemma}
\begin{proof}
\eqref{3} follows immediately by summing the definitions of the
$\phi_i$'s.  \eqref{1} and \eqref{2} follow immediately by plugging
in the definition of $\Ad$ \eqref{pi}.

\eqref{4} is not as trivial but can be derived from the defining
equations of $\SOT$ by a simple Gr\"obner Basis calculation.

To see \eqref{5} observe that by \eqref{3} at least one of
the $\phi_i(M)$'s is always nonzero and by \eqref{4} all the
nonzero $\phi_i(M)$'s always agree once one mods out by squares.

\eqref{6} follows since by \eqref{3} at each prime of $R^+$ at
least one of the $\phi_i(M)$'s must have valuation $0$.  The total
positivity follows from the definitions of $\phi_i$ and the fact that
for $M\in\SOT(\R)$ we always have $\Tr(M)\ge-1$.

Finally, \eqref{7} holds, because the sets 
$\{\phi_j(M)\}$, $\{\phi_j(MM_i)\}$, $\{\phi_j(M_iM)\}$ for $1 \le j \le 4$ are visibly equal.
\end{proof}

\begin{theorem}
The map $\phi:\SOT(R^+)\to\Sel^+_2(R^+)$ is group homomorphism and
$$1\to\PSUT(R)\xrightarrow{\Ad}\SOT(R^+)\xrightarrow{\phi}\Sel^+_2(R^+)$$
  is an exact sequence.
\end{theorem}
\begin{proof}
  In view of Lemma \ref{phi theta}\eqref{7},
  we may assume that $\phi_1(MN) \ne 0$.
It can be checked using a simple Gr\"obner basis calculation that for
$M,N\in\SOT(R^+)$ and $1 \le i \le 4$, the equation
\begin{align}
\label{pinto}
\phi_i(M)\phi_i(N)&\phi_1(MN)=\\
&\big(\phi_i(M)\phi_i(N) \pm
\theta_{ij}(M)\theta_{ij}(N) \pm \theta_{ik}(M)\theta_{ik}(N) \pm
\theta_{i\ell}(M)\theta_{i\ell}(N)\big)^2\nonumber
\end{align}
follows from the defining equations of
$\SOT$, where $\{i,j,k,\ell\} = \{1,2,3,4\}$ and the sign is $-1$ when
$1$ appears in the subscript and $1$ otherwise.
Hence $\phi(M)\phi(N)=\phi(MN)$ as long as 
$\phi_i(M)$, $\phi_i(N)$ are both nonzero for the same $i$.

If three of the $\phi_i(M)$ are $0$, then $M \in \{I_3,M_x,M_y,M_z\}$ and
it is simple to check that $\phi(MN) = \phi(M)\phi(N)$; similarly if three of
the $\phi_i(N)$ are $0$.  Otherwise, there is no problem unless
two are $0$ for $M$ and the other two are $0$ for $N$.
Suppose that
$\phi_1(M) = \phi_2(M) = \phi_3(N) = \phi_4(N)$ (the other cases are similar).
Then we have
$$M = \begin{pmatrix}a&b&0\\b&-a&0\\0&0&-1\end{pmatrix}, \quad
N = \begin{pmatrix}c&d&0\\-d&-c&0\\0&0&1\end{pmatrix},$$
where $a^2 + b^2 = c^2 + d^2 = 1$.  In this case it is easy to check that
$\phi_1(MN) = 0$, contradicting our choice of $N$ (and it is also easy to check
that $\phi(MN) = \phi(M) \phi(N)$).
Thus, $\phi$ is a
group homomorphism.

That $\phi\circ\Ad=1$ follows from \ref{phi theta}(\ref{1}).  Now
suppose $M\in\ker\phi$, hence the $\phi_i(M)$ are all squares in
$R^+$.  Let $a_i=\sqrt{\phi_i(M)}$ with signs chosen so that
$a_ia_j=\theta_{ij}(M)$; we can do this by \ref{phi theta}(\ref{2}).
Now it is again straightforward to check that the
equations $$M=\Ad\left(\left[\begin{array}{cc} a_1 + a_2i & a_3+ a_4i
    \\ -a_3 + a_4i & a_1-a_2i \end{array}\right]\right)$$ follow from
the defining equations of $\SOT$.
\end{proof}

\begin{proof}
It can be checked using a simple Gr\"obner Basis calculation that for
$M,N\in\SOT(R^+)$ the equation
\begin{align}
\label{pinto}
\phi_1(M)\phi_1(N)&\phi_1(MN)=\\
&\big(\phi_1(M)\phi_1(N) -
\theta_{12}(M)\theta_{12}(N) - \theta_{13}(M)\theta_{13}(N) -
\theta_{14}(M)\theta_{14}(N)\big)^2\nonumber
\end{align}
 follows from the defining equations of
$\SOT$.  Hence, $\phi(M)\phi(N)=\phi(MN)$ as long as all three $\phi_1(M)$,
$\phi_1(M)$, $\phi_{1}(MN)$ are
nonzero.  In general, one has give expressions analogous to 
\eqref{pinto} for $\phi_i(M)$, $\phi_j(M)$, $\phi_k(M)$ with $1\leq i,j,k\leq 4$.
These additional cases are similar.
Thus, $\phi$ is a
group homomorphism.

That $\phi\circ\Ad=1$ follows from \ref{phi theta}(\ref{1}).  Now
suppose $M\in\ker\phi$, hence the $\phi_i(M)$ are all squares in
$R^+$.  Let $a_i=\sqrt{\phi_i(M)}$ with signs chosen so that
$a_ia_j=\theta_{ij}(M)$, we can do this by \ref{phi theta}(\ref{2}).
Now it is again straightforward to check that the
equations $$M=\Ad\left(\left[\begin{array}{cc} a_1 + a_2i & a_3+ a_4i
    \\ -a_3 + a_4i & a_1-a_2i \end{array}\right]\right)$$ follow from
defining equations of $\SOT$.
\end{proof}
\begin{corollary}
\label{foal}
The subgroups $\PSUT(R)$ and $\PUT(R)$ of $\SOT(R^+)$ are normal.
\end{corollary}
\begin{definition1}
\label{mare}
{\rm
\begin{enumerate}[\upshape (a)]
\item
\label{mare1}
Set $C(R)=\SOT(R^+)/\PSUT(R)$, $\overline{C}(R)=\SOT(R^+)/\PUT(R)$,
$c(R)=\# C(R)$, and $\overline{c}(R)=\# \overline{C}(R)$.
Hence $c(R)$, $\overline{c}(R)$ are powers of 2 with $c(R)\neq 1$.
We have
\begin{equation}
\label{jay}
c(R)=[\PUT(R):\PSUT(R)]\overline{c}(R).
\end{equation}
Example \ref{dreary} shows that $\overline{c}(\Z[\sqrt{21},i,1/2])\neq 1$.
\item
\label{mare2}
Let $r(n)$ be the number of primes in $K_n:=\Q(\zeta_n)$ above $2$ and $r_+(n)$
be the number of primes in $F_n:=\Q(\zeta_n)^+$ above $2$.
\end{enumerate}
}
\end{definition1}

We now state a result from \cite{IJKLZ2} which we will
need:
\begin{proposition}\textup{(\cite[Proposition 2.3]{IJKLZ2})}
\label{index}
Suppose $n\geq 8$ and $4|n$ with $r(n)$, $r_+(n)$ as in Definition 
\textup{\ref{mare}\eqref{mare2}}.
\begin{enumerate}[\upshape (a)]
\item
\label{index1}
$\PUT(R_n)/\PSUT(R_n)\cong (\Z/2\Z)^{1+r(n)-r_+(n)}$.
\item
\label{index2}
$\PUT(R_n)/\PUTz(R_n)\cong (\Z/2\Z)^{r(n)-r_+(n)}$.
\item
\label{index3}
$\PUTz(R_n)/\PSUT(R_n)\cong\Z/2\Z$.
\end{enumerate}
\end{proposition}

Combining Proposition \ref{index}\eqref{index1} with \eqref{jay}
then gives the following:

\begin{proposition}
\label{lint}
For $R_n=\Z[\zeta_n, 1/2]$, $4|n$, $n\geq 8$, we have
$c(R_n)=2^{1+r(n)-r_+(n)}\overline{c}(R_n)$.
\end{proposition}

We can compute $c(R)$ and $\overline{c}(R)$ in
some important examples with $R=R_n:=\Z[\zeta_n, 1/2]$.
\begin{theorem}
\label{ex}
\begin{enumerate}[\upshape (a)]
\item
\label{ex1}
Suppose $n=2^s$, $n\geq 8$.
Then $c(R_n)=2$ and $\overline{c}(R_n)=1$.
\item
\label{ex2}
Suppose $n=3\cdot 2^s$, $s\geq 2$.  
Then $c(R_n)=2$ and $\overline{c}(R_n)=1$.
\end{enumerate}
\end{theorem}
\begin{proof}
Suppose $n=2^s$ or $n=3\cdot 2^s$, $n\geq 8$.  Then there is one
prime in $K_n=\Q(\zeta_n)$ above $2$ and $r(n)=r_+(n)=1$.  Hence by
Proposition \ref{lint}, $c(R_n)=2\overline{c}(R_n)$.  But
by Proposition \ref{hawk}, $\Sel_2^+(R_n^+)\cong \Z/2\Z$.
Hence $c(R_n)\leq 2$ and we therefore must have
$c(R_n)=2$ and $\overline{c}(R_n)=1$.

\end{proof}

\section{Amalgamated products and the Clifford-cyclotomic group}
\label{yellow}

%We can extend the adjoint action $\Ad$ (\ref{pi}) to a group homomorphism $\pi:\UT(R) \rightarrow \SOT(R^+)$ via conjugation on the trace-$0$ $2\times 2$ hermitian matrices in the Pauli basis. Clearly, $\pi(g%) = \Ad \left(\frac{1}{\sqrt{\det g}}g \right)$ for an arbitrary choice of $\sqrt{\det g} $. 
%\begin{definition1}
%\label{beast}
%The map $\pi$ induces an injection $\overline{\pi}:\PUT(R)\rightarrow \SOT(R^+)$, which need
%not be surjective.  Define 
%$C(R)=\SOT(R^+)/\overline{\pi}(\PUT(R))$ and 
%$c(R)=\# C(R)$.   

%\end{definition1}

Set
$\overline{\G}_n = \pi(\Gn) \subseteq \SOT(R_n^+)$ and 
$\overline{\Ss\!\G}_n = \Ad(\Ss\!\Gn) \subseteq \SOT(R_n^+)$.

\begin{theorem}\label{amalgam}
Assume $4|n$, $n\geq 8$.  We have
\[
\Pro \Gn\cong\overline{\G}_n=G(4,n)\cong S_4*_{D_4}D_n.
\]
\end{theorem}
\begin{proof}
By Proposition \ref{valley3}\eqref{valley32} and Theorem \ref{gone}\eqref{gone3} it suffices to show that $\pi(H_nT^{2m})$ and $\pi(T_n)$ are rotations of order 4 and $n$ about orthogonal axes. Now
$$T^{2m} = \left[ \begin{matrix}1 & 0 \\ 0 & -1 \end{matrix} \right],\quad
\mbox{so}\quad
HT^{2m} =  \frac{1}{2}\left[ \begin{matrix}1+i & -1-i \\ 1+i & 1+i \end{matrix} \right] $$
which has determinant  $i = \zeta_8^2$. Define 
$$\widetilde{H} = \frac{1}{\zeta_8} HT^{2m} = \frac{1}{2} \left[ \begin{matrix} \zeta_8 - \zeta_8^3 & \zeta_8^3 - \zeta_8 \\ \zeta_8-\zeta_8^3 & \zeta_8-\zeta_8^3\end{matrix} \right] = \frac{1}{\sqrt{2}}\left[ \begin{matrix} 1 & -1 \\ 1 & 1 \end{matrix}\right]\, .$$
We calculate using (\ref{pi}):
$$\pi(HT^{2m}) = \Ad(\widetilde{H}) = \left[ \begin{matrix} 0 & 0 & 1 \\ 0 & 1 & 0 \\ -1 & 0 & 0 \end{matrix}\right] = \left[ \begin{matrix} \cos(\pi/2) & 0 & \sin (\pi/2) \\ 0 & 1 & 0 \\ -\sin(\pi/2) & 0 & \cos(\pi/2) \end{matrix}\right]\, ,$$
which is a rotation around the $y$-axis by $\pi/2$, while 
$$\pi(T_n) = \Ad\left(\left[\begin{matrix} \zeta_{2n}^{-1} & 0 \\ 0 & \zeta_{2n}\end{matrix}\right]\right) = \Ad\left(\left[\begin{matrix} \cos(\pi/n) -i \sin(\pi/n) & 0 \\ 0 & \cos(\pi/n) + i \sin(\pi/n)\end{matrix}\right]\right)=$$
$$\left[\begin{matrix}\cos(2\pi/n) & -\sin(2\pi/n) & 0\\ \sin(2\pi/n) & \cos(2\pi/n) & 0 \\ 0 & 0 & 1 \end{matrix}\right]$$
is a rotation by $2\pi/n$ about the $z$-axis.
\end{proof}
 The finite subgroups of $\SUT(\CC)$ are well-known; see
\cite[Th\'{e}or\`{e}me I.3.7]{v}.
Let $D_n$ be the dihedral group of order $2n$.
Denote by $E_{48}$ the tetrahedral group, i.e., the degree-$2$ central
extension of $S_4$, and by  $Q_{4n}$ the quaternion group of
order $4n$ (called {\em dicyclique} in \cite{v}).  
We have $Q_{4n}/\langle \pm 1
\rangle\cong D_n$.
\begin{corollary}
\label{peachy}
Let $\HH$ be the pullback of $\overline{\G}_n$
under the surjective map $\SUT(\CC)\stackrel{\Ad}{\rightarrow}\SOT(\R)$.
Then $\Ss\!\Gn\subseteq \HH$ with $[\HH:\Ss\!\Gn]=2$ and
\begin{equation}
\label{riddle}
\HH\cong E_{48}\ast_{Q_{16}}Q_{4n} .
\end{equation}
\end{corollary}
\begin{proof}
We have $\Pro\Gn/\Pro\Ss\!\Gn\cong \mu_n/\mu_n^2\cong \Z/2\Z$
since $n$ is even.  But 
\begin{equation}
\label{rested}
\Pro\Gn\cong \pi(\Gn)=\overline{\G}_n=G(4,n)
\cong S_4\ast_{D_4}D_n
\end{equation}
and $\Pro\Ss\!\Gn\cong \Ad(\Ss\!\Gn)=\overline{\Ss\!\G}_n$. Hence
$\Ad^{-1}(\overline{\G}_n):=\HH\cong E_{48}\ast_{Q_{16}}Q_{4n}$
and $\Ad^{-1}(\overline{\Ss\!\G}_n)=\Ss\!\Gn$.
Since $[\overline{\G}_n:\overline{\Ss\!\G}_n]=2$, it follows
that $[\HH:\Ss\!\Gn]=2$.

\end{proof}

\section{Euler-Poincar\'{e} characteristics}
\label{tall}

In this section we determine the Euler-Poincar\'{e}
characteristics of unitary groups over cyclotomic rings
and Clifford-cyclotomic groups.  These results will then be
used in  the proof of Theorem \ref{green}.
General references for Euler-Poincar\'{e} characteristics
are \cite[Chapter 9]{b} and \cite{s1}.

\begin{definition1}[{\cite[Section IX.6]{b}}]
\label{fht}
  A group $\Gamma$ is of {\em finite homological type} if $\Gamma$ has
  finite virtual cohomological dimension and, for every $\Gamma$-module $M$
  that is finitely generated as an abelian group and every natural number $i$,
  the homology group $H_i(\Gamma,M)$ is finitely generated.
\end{definition1}

\begin{proposition}
\label{snore}
\begin{enumerate}[\upshape (a)]
\item
\label{snore1}
Suppose 
$1\rightarrow \Gamma'\rightarrow \Gamma\rightarrow \Gamma''\rightarrow 1$
is a short exact sequence of groups with $\Gamma'$,
$\Gamma''$ of finite homological type.
If $\Gamma$ is virtually torsion-free, then $\Gamma$ is of finite
homological type and 
\[
\chi(\Gamma)=\chi(\Gamma')\chi(\Gamma'').
\]

\item
\label{snore2}
Suppose $\Gamma'$ is a subgroup of $\Gamma$ of finite index and
$\chi(\Gamma)$ is defined.  Then $\chi(\Gamma')$ is defined and
$\chi(\Gamma')=[\Gamma:\Gamma']\chi(\Gamma)$.
\item
\label{snore3}
Suppose $\Gamma'\leq \Gamma$ with
$\chi(\Gamma)$ and $\chi(\Gamma')$ both defined.
If $|\chi(\Gamma')|/|\chi(\Gamma)|$ is not a positive integer,
then $\Gamma'$ has infinite index in $\Gamma$.  In particular
this holds if $|\chi(\Gamma')|<|\chi(\Gamma)|$.
\item
\label{snore4}
Suppose $\Gamma'$ and $\Gamma''$ are finite groups with 
$A\leq \Gamma'$ and $A\leq \Gamma''$.
Let $\Gamma=\Gamma'\ast_{A}\Gamma''$.  Then
\[
\chi(\Gamma)=\frac{1}{\#\Gamma'} +\frac{1}{\#\Gamma''}-\frac{1}{\# A}.
\]
\end{enumerate}
\end{proposition}
\begin{proof}
(a), (b) are parts (d) and (c) of \cite[Proposition 7.3]{b}.
(c) is an immediate consequence of (a), while
(d) is \cite[Corollaire 1, p.~104]{s1}.
\end{proof}
\begin{theorem}
\label{mouse}
Assume $4|n$.  Then $\chi (\Ss\!\Gn)= 
\chi(G(4,n))=\chi(\Pro\Gn)
=-\frac{1}{12} + \frac{1}{2n}$.
\end{theorem}
\begin{proof}
Let $\HH$ be as in Corollary \ref{peachy}.
By \eqref{riddle} and Proposition \ref{snore}\eqref{snore4} we have
\begin{equation*}
\begin{split}
\chi(\HH) = \chi(E_{48}\ast_{Q_{16}}Q_{4n}) &= \frac{1}{\#E_{48}} +
\frac{1}{\#Q_{4n}} - \frac{1}{\#Q_{16}} \\ &= \frac{1}{48} + \frac{1}{4n} -
\frac{1}{16} = -\frac{1}{24} + \frac{1}{4n}.
\end{split}
\end{equation*}
But $\Ss\!\Gn$ is an index-$2$ subgroup of $\HH$ from Corollary \ref{peachy}, 
so by Proposition \ref{snore}\eqref{snore2}
\[
\chi (\Ss\!\Gn)=
2\chi(\HH)= -\frac{1}{12} + \frac{1}{2n}.
\]
We have $\chi(G(4,n))=\chi(\Pro\Gn)
=-1/12 + 1/2n$ from \eqref{rested}
and Proposition \ref{snore}\eqref{snore4}.
\end{proof}

We will need the following in the proof of Theorem \ref{plants} below.
\begin{remark1}\label{rem:def-reductive} Recall that a connected linear
 algebraic group $G$ over a perfect field is
  {\em reductive} if it admits a representation with finite kernel that is
  a direct sum of irreducible representations.  An alternative definition
  sufficient for this paper is that $G$ over an algebraically closed field
  is reductive if and only if every smooth connected unipotent normal
  subgroup of $G$ is trivial, and if $k$ is perfect then $G$ is
  reductive over $k$ if and only if it is over $\bar k$.
\end{remark1}

\begin{definition1}
\label{garden}
{\rm 
Set
\[
M_n:=2^{1-[F_n:\Q]}\left|\zeta_{F_n}(-1)\right|\prod_{\fp|2}\left| 1-\Nm_{F_n/\Q}(\fp )\right|.
\]
 }
\end{definition1}

\begin{theorem}
\label{plants}
Suppose $n\geq 8$ and $4|n$ with $r(n)$, $r_+(n)$ as in Definition \textup{\ref{mare}\eqref{mare2}}.
\begin{enumerate}[\upshape (a)]
\item
\label{most}
$\chi(\SUT(R_n))=-M_n/2$.
\item
\label{host}
$\chi(\PSUT(R_n))=2\chi(\SUT(R_n))=-M_n$.
\item
\label{ghost}
$\chi(\PUTz(R_n))=\chi(\SUT(R_n))=-M_n/2$.
\item
\label{toast}
\[
\chi(\PUT(R_n))=\frac{\chi(\SUT(R_n))}{2^{r(n)-r_+(n)}}=
\frac{\chi(\PSUT(R_n))}{2^{1+r(n)-r_+(n)}}=-\frac{M_n}{2^{1+r(n)-r_{+}(n)}}.
\]
\item
\label{roast}
Put $c_n=c(R_n)$ and 
$\overline{c}_n=\overline{c}(R_n)$ 
as in Definition \textup{\ref{mare}}.
Then 
\begin{equation}
\label{boast}
\chi(\SOT(R_n^+))=\chi(\PUT(R_n))/\overline{c}_n
=-\frac{M_n}{2^{1+r(n)-r_+(n)}\overline{c}_n}
=\chi(\PSUT(R_n))/c_n
=-\frac{M_n}{c_n}.
\end{equation}
\end{enumerate}
\end{theorem}
\begin{proof}
\eqref{most} follows from a result of Harder \cite[Section 3.7, (*)]{s1}.

The claims \eqref{host}, \eqref{ghost}, \eqref{toast}, \eqref{roast}
 are obtained by combining this with Proposition \ref{snore},
once we verify that all groups involved are of finite homological type.
We start from the fact, due to Borel and Serre \cite[page 218]{b}, that a
torsion-free reductive $S$-arithmetic group $\Gamma$ is of type FL 
(i.e., that the
$\Z\Gamma$-module $\Z$ has a finite free resolution).  This is stated for arithmetic
groups over $\Q$; however, the result follows more generally by restriction
of scalars.  
The only problem is to verify that restriction of scalars
preserves reductivity.  As noted in Remark \ref{rem:def-reductive}, 
in characteristic $0$ an algebraic group
is reductive over $k$ if and only if it is reductive over $\bar k$, and
for a finite extension $K/F$ we have
$\Res_{K/F} G \otimes_F \bar F \equiv (G \otimes_K \bar F)^{[K:F]}$.  A finite
direct product of reductive groups is clearly reductive, so it follows that
if $G$ is reductive, then so is $\Res_{K/F} G$.  As $\SUT$ is a simple group,
it is certainly reductive.
This shows that our groups are all VFL
(as usual, the subgroup of matrices congruent to $1$ modulo a large prime is torsion-free).

Since free modules are projective, VFL implies VFP, and groups of type VFP are of finite
homological type.  This is enough to apply \cite[Proposition 7.3]{b}.

\eqref{host}: Apply Proposition \ref{snore}\eqref{snore1} to 
\[
1\longrightarrow \langle\pm 1\rangle\longrightarrow \SUT(R_n)\longrightarrow
\PSUT(R_n)\longrightarrow 1 ,
\]
using the fact that $\SUT(R_n)$ is virtually torsion-free because it
is arithmetic (so a sufficiently small congruence subgroup is torsion-free).

\eqref{ghost}: Apply Proposition \ref{snore}\eqref{snore1} to
\[
1\longrightarrow \PSUT(R_n)\longrightarrow \PUTz(R_n)
\longrightarrow \Z/2\Z\longrightarrow 0 
\]
from Proposition \ref{index}\eqref{index3}.  To show that $\PUTz(R_n)$
is virtually torsion-free, it suffices to show that the finite-index
subgroup $\PSUT(R_n)$ is virtually torsion-free.  But $\PSUT(R_n)$
is virtually torsion-free since it is a finite quotient 
$\SUT(R_n)$, which is virtually torsion-free from \eqref{host}.

\eqref{toast}: Apply Proposition \ref{snore}\eqref{snore1} to 
\[
1\longrightarrow \PSUT(R_n)\longrightarrow \PUT(R_n)\longrightarrow
(\Z/2\Z)^{1+r(n)-r_+(n)}\longrightarrow 0
\]
from Proposition \ref{index}\eqref{index1}.  The group
$\PUT(R_n)$ is virtually torsion-free because its finite-index
subgroup $\PSUT(R_n)$ is virtually torsion-free from \eqref{ghost}.

\eqref{roast}: Apply Proposition \ref{index}\ref{index1}
to
\begin{align*}
1 &\longrightarrow  \PUT(R_n)\stackrel{\overline{\pi}}{\longrightarrow}
\SOT(R_n^+)\longrightarrow \overline{C}(R_n)\longrightarrow 1\text{ and}\\
& 1\longrightarrow \PSUT(R_n)\stackrel{\Ad}{\longrightarrow}
\SOT(R_n^+)\longrightarrow C(R_n)\longrightarrow 1,
\end{align*}
as in Definition \ref{mare}; $\#\overline{C}(R_n)=\overline{c}(R_n)=\overline{c}_n$ and $\#C(R_n)=c(R_n)=c_n$.  The group $\SOT(R_n^+)$
is virtually torsion-free since it is arithmetic.

\end{proof}
\begin{remark1}
\label{dream}
{\rm
Suppose $n=2^s\geq 8$. Then $r(n)=r_+(n)=1$.
Serre \cite[p.~48]{s2} uses Tamagawa numbers to show in this case that
$\chi(\SOT(R_n^+))=-M_n/2$ as in Theorem \ref{gut}\eqref{gut2}.  
Theorem \ref{plants}\eqref{roast} then shows that $c(R_n)=2$
and $\overline{c}(R_n)=1$, 
giving an independent proof of Theorem \ref{ex}\eqref{ex1}.
}
\end{remark1}

\section{Proof of Theorem \texorpdfstring{\ref{green}}{1.2}}
\label{pork}

We first prove Theorem \ref{green}\eqref{green2}.  
%Then Theorem \ref{green}(1) follows from it by Theorem \ref{mountain}.
It is already known that $\Ss\!\Gn=\SUT(R_n)$ for $n=8, 12, 16, 24$
(Theorem \ref{gone}\eqref{gone1}).
We will prove that $\Ss\!\Gn$ is not a finite-index subgroup of
$\SUT(R_n)$ otherwise.  By Proposition \ref{snore}\eqref{snore2},
to do this it suffices to show 
$\lvert\chi(\Ss\!\Gn)\rvert<\lvert\chi(\SUT(R_n))\rvert$ 
for $n\notin\{8, 12,16,24\}$

Let $S$ be the places of $F_n=K^+_n$ above $2\infty$ and denote
by $\zeta_{F_n, S}(s)$ the Dedekind zeta function of $F_n$ with the
Euler factors at finite places in $S$ omitted.  Then the
Euler-Poincar\'e characteristic of $\Gamma_n=\SUT(R_n)$ is given in
\cite[Section 3.7]{s1}:
\begin{align}
\label{snail}
\lvert\chi(\Gamma_n)\rvert =
2^{-[F_n:\Q]}\lvert\zeta_{F_n,S}(-1)\rvert =&
2^{-[F_n:\Q]}\lvert\zeta_{F_n}(-1)\rvert\prod_{\fp|2}\lvert1-N_{F_n/\Q}(\fp)\rvert\\
\nonumber \geq &  2^{-[F_n:\Q]}\lvert\zeta_{F_n}(-1)\rvert.
\end{align}
 By the functional
equation for $\zeta_{F_n}$,
$$\lvert\zeta_{F_n}(-1)\rvert =
\zeta_{F_n}(2)\lvert\Disc(F_n)\rvert^{3/2}(2\pi^2)^{-[F_n:\Q]}$$
and
by \cite[Proposition 2.7]{w}
$$\lvert\Disc(K_n)\rvert =
\frac{n^{\phi(n)}}{\prod_{p|n}p^{\phi(n)/(p-1)}}.$$
As for $F_n$, let $f=\sqrt{\lvert N_{F_n/\Q}\Disc(K_n/F_n)\rvert}$. 
Then $f=1$
unless $n$ is a power of $2$, in which case $f=2$.  Now we have
$$\lvert\Disc(F_n)\rvert =
\sqrt{\frac{\lvert\Disc(K_n)\rvert}{\lvert N_{F_n/\Q}\Disc(K_n/F_n)\rvert}} =
\frac{n^{\phi(n)/2}}{f\prod_{p|n}p^{\phi(n)/(2(p-1))}}$$  
using standard properties of the discriminant in towers
\cite[Corollary 2.10, p.~202]{neu}.

Hence, 
\begin{equation*}
\begin{split}
\lvert\chi(\Gamma_n)\rvert& \geq 2^{-[F_n:\Q]}\lvert\zeta_{F_n}(-1)\rvert=\zeta_{F_n}(2)\lvert\Disc(F_n)\rvert^{3/2}(2\pi)^{-2[F_n:\Q]}\\
&>\lvert\Disc(F_n)\rvert^{3/2}(2\pi)^{-2[F_n:\Q]}=\left(\frac{n^{\phi(n)/2}}{f\prod_{p|n}p^{\phi(n)/(2(p-1))}}\right)^{3/2}(2\pi)^{-2[F_n:\Q]}\\
&=\frac1{f^{3/2}}\left(\left(\frac{n}{\prod_{p|n}p^{1/(p-1)}}\right)^{3/2}(2\pi)^{-2}\right)^{[F_n:\Q]}\\
&>\frac1{2^{3/2}}\left(\left(\frac{n}
{2\prod_{p|n,\,p>2}p^{1/2}}\right)^{3/2}(2\pi)^{-2}\right)^{[F_n:\Q]}\\
&>\frac1{2^{3/2}}\left(\left(\frac{n}{2(n/4)^{1/2}}\right
)^{3/2}(2\pi)^{-2}\right)^{[F_n:\Q]}
=\frac{\left(n^{3/4}(2\pi)^{-2}\right)^{[F_n:\Q]}}{2\sqrt{2}} ,
\end{split}
\end{equation*}
which is greater than $\frac1{2\sqrt{2}}$ and hence greater than
$\frac{1}{12} - \frac{1}{2n}$ as long as $n>134.5>(2\pi)^{8/3}$.
If $n\notin\{8,12,16,24\}$ , $4|n$, and $8<n\leq 132$, then
it can be manually checked from  \eqref{snail} that we still have
$|\chi(\Gamma_n)|>\frac{1}{12}-\frac{1}{2n}$.
Reassuringly $\chi(\Gamma_n)=1/12 -1/(2n)$ for $n=8, 12, 16, 24$.
Hence $[\SUT(R_n):\Ss\!\Gn]=\infty$ for $4|n$,
$n\geq 8$, and $n\notin\{8, 12, 16, 24\}$ by Proposition \ref{snore}\eqref{snore2}, proving
Theorem \ref{green}\eqref{green2}.

To prove Theorem \ref{green}\eqref{green1},
note that $n=[\Gn:\Ss\!\Gn]=[\UTz(R_n):\SUT(R_n)]$
by Proposition \ref{valley3}\eqref{valley31} and \eqref{valley4}.
Hence $[\SUT(R_n):\Ss\!\Gn]=[\UTz(R_n):\Gn]$, and 
so Theorem \ref{green}\eqref{green2} together with Theorem \ref{gone}\eqref{gone1}
implies Theorem \ref{green}\eqref{green1}.

The proof of Theorem \ref{green}\eqref{green3} is similar:
both surjections 
$\Gn\twoheadrightarrow\Pro\Gn \cong \overline{\G}_n$ 
and $\UTz(R_n)\twoheadrightarrow \PUTz(R_n)$ 
have kernel of order $n$.  Hence $[\UTz(R_n):\Gn]=[\PUTz(R_n):\Pro\Gn]$.
But then we have
\[
\pi(\Gn)=G(4,n)\cong  \Pro\Gn\subseteq
\pi(\UTz(R_n))\cong \PUTz(R_n)\subseteq \SOT(R_n^+).
\]
Hence $[\UTz(R_n):\Gn]=\infty$ implies $[\SOT(R_n^+):G(4,n)]=\infty$.
Theorem \ref{green}\eqref{green3} then follows from Theorem \ref{green}
\eqref{green1}
and Theorem \ref{gone}\eqref{gone1}, concluding the proof of Theorem
\ref{green}.

\bibliographystyle{plain}
\bibliography{CCf}%../local,outside
\end{document}